\documentclass[reqno]{amsart}
\usepackage{mathrsfs}
\usepackage{amssymb}
\usepackage{amscd}
\theoremstyle{plain}
\newtheorem{theorem}{Theorem}[section]
\newtheorem{lemma}{Lemma}[section]

\begin{document}
\title{On rational functional identities involving inverses on matrix rings}
\thanks{$^*$ Corresponding author}
\maketitle
\begin{center}
Yingyu Luo\\
College of Mathematics, Changchun Normal University, Changchun
130032, China\\
E-mail: luoyingyu1980@163.com\\
Qian Chen\\
Department of Mathematics, Shanghai Normal University,
Shanghai 200234, China.\\
Email address: qianchen0505@163.com\\
Yu Wang$^*$\\
Department of Mathematics, Shanghai Normal University,
Shanghai 200234, China.\\
Email address: ywang2004@126.com
\end{center}
\maketitle

\begin{abstract}
Let $n\geq 3$ be an integer. Let $\mathcal{D}$ be a division ring with char$(\mathcal{D})>n$ or char$(\mathcal{D})=0$. Let $\mathcal{R}=M_m(\mathcal{D})$ be a ring of $n\times n$ matrices over $D$, $m\geq 2$. The main theorem in the paper states that the only additive
maps $f$ and $g$ satisfying that $f(X)+X^ng(X^{-1})=0$ for all invertible $X\in \mathcal{R}$, are zero maps, which generalizes both a result proved by Dar and Jing and a result proved by  Catalano and Merch$\acute{a}$n.
\end{abstract}

\section{Introduction}
Let $\mathcal{R}$ be an associative ring. Throughout the paper we will denote by $\mathcal{R}^{\times}$ the set of all invertible elements of $\mathcal{R}$. There has been a great interest in the study of functional identities on rings in the last decades (see \cite{Bre} for details). Most
functional identities deal with identities involving arbitrary elements in a ring. It
seems that the first result on identities involving inverses of elements is due to
Vukman who proved the following result in 1987.

\begin{theorem}\cite[Theorem]{Vuk}
Let $\mathcal{D}$ be a skew field of characteristic not
two and let $f:\mathcal{D}\rightarrow \mathcal{D}$ be an additive mapping such that
\[
f(x)+x^2f(x^{-1})=0
\]
for all $x\in \mathcal{D}^{\times}$. Then $f=0$.
\end{theorem}

In 2018, Catalano \cite{Cat} considered the following rational identity in
the context of a division ring $\mathcal{D}$:
\[
f(x)x^{-1}+xg(x^{-1})=0
\]
for every $x\in \mathcal{D}^{\times}$, and where $f,g:\mathcal{D}\rightarrow \mathcal{D}$ are additive maps. She proved that $f(x)=xq+d(x)$ and $g(x)=-qx+d(x)$,
where $d:\mathcal{D}\rightarrow \mathcal{D}$ is a derivation and $q\in \mathcal{D}$ is a fixed element. In the same year, Catelano \cite{Cat2} extended the main result in \cite{Cat}. In 2020, Argac etc \cite{Lee} generalized the main results in \cite{Cat,Cat2}.

In 2023,  Dar and Jing \cite{Jing} generalized Vukman' result and obtained the following two results:

\begin{theorem}\cite[Theorem 1.4]{Jing}
Let $\mathcal{D}$ be a division ring which is not a field with characteristic
different from $2$ and let $f,g:\mathcal{D}\rightarrow \mathcal{D}$ be additive maps satisfying the identity
\[
f(x)+x^2g(x^{-1})=0
\]
for all $x\in \mathcal{D}^{\times}$. Then $f(x)=xq$ and $g(x)=-xq$ for all $x\in \mathcal{D}$, where $q\in \mathcal{D}$ is a fixed element.
\end{theorem}

\begin{theorem}\cite[Theorem 1.3]{Jing}\label{T1.3}
Let $\mathcal{D}$ be a division ring which is not a field with characteristic
different from $2$ and $3$, and $\mathcal{R}=M_n(\mathcal{D})$ with $n>1$. Let $f,g:\mathcal{R}\rightarrow \mathcal{R}$ be additive maps satisfying the identity
\[
f(X)+X^2g(X^{-1})=0
\]
for all $X\in \mathcal{R}^{\times}$. Then there exists an element $P\in \mathcal{R}$ such that $f(X)=XP$ and $g(X)=-XP$ for all $X\in \mathcal{D}$.
\end{theorem}

Recently, Catalano and Merch$\acute{a}$n \cite{Cat1} generalized Theorem 1.2 and obtained the following result.

\begin{theorem}\cite[Theorem 1]{Cat}\label{C}\label{T1.4}
Let $\mathcal{D}$ be a division ring with char$(\mathcal{D})\neq 2,3$, and let $f,g:\mathcal{D}\rightarrow \mathcal{D}$ be additive maps satisfying the identity
\[
f(x)+x^ng(x^{-1})=0
\]
for all $x\in \mathcal{D}^{\times}$, where $n=3$ or $n=4$. Then $f=0=g$.
\end{theorem}

In the present paper we shall generalize both Theorem \ref{T1.3} and Theorem \ref{T1.4}. More precisely, we shall prove the following two results.

\begin{theorem}\label{T1}
Let $n\geq 3$ be an integer. Let $\mathcal{D}$ be a division ring  with char$(\mathcal{D})>n$ or char$(\mathcal{D})=0$. Let $f,g:\mathcal{D}\rightarrow \mathcal{D}$ be additive maps satisfying the identity
\begin{equation}\label{e1}
f(x)+x^ng(x^{-1})=0
\end{equation}
for all $x\in \mathcal{D}^{\times}$. Then $f=0=g$.
\end{theorem}
\begin{theorem}\label{T2}
Let $n\geq 3$ be an integer. Let $\mathcal{D}$ be a division ring with char$(\mathcal{D})>n$ or char$(\mathcal{D})=0$. Let $\mathcal{R}=M_m(\mathcal{D})$, $m>1$, and let $f,g:\mathcal{R}\rightarrow \mathcal{R}$ be additive maps satisfying the identity
\[
f(X)+X^ng(X^{-1})=0
\]
for all $X\in \mathcal{R}^{\times}$. Then $f=0=g$.
\end{theorem}

We organize this paper as follows: In Section $2$ we shall give some technical results, which will be used in the proof of our main result. In Section $3$ we shall give the proof of Theorem \ref{T1}. In Section $4$ we shall give the proof of Theorem \ref{T2}.

\section{Some technical results}

By $\mathcal{Z}$ we denote the set of all integers. Let $s\geq 1$ be an integer. We set
\[
\mathcal{N}_s=\{1,2,\ldots,s\}.
\]

The following technical result will be used in the proofs of both Theorem \ref{T1} and Theorem \ref{T2}.

\begin{lemma}\label{L2.1}
Let $\mathcal{D}$ be a division ring with char$(\mathcal{D})>t$, $t\geq 1$ or char$(\mathcal{D})=0$. Let $\mathcal{R}=M_m(\mathcal{D})$, where $m>1$. Let $p(x)=a_0x^t+a_1x^{t-1}+\cdots +a_t$ be a polynomial over $\mathcal{Z}$ with char$(\mathcal{D})\nmid a_0$. Let $A\in M_m(\mathcal{D})$. Suppose that $p(\alpha)A=0$ for all $\alpha\in \mathcal{N}_{t+1}$. Then $A=0$. In particular, for $c\in \mathcal{D}$, $p(\alpha)c=0$ for all $\alpha\in \mathcal{N}_{t+1}$, then $c=0$.
\end{lemma}
\begin{proof}
By our hypotheses we get
\[
(p(\alpha+1)-p(\alpha))A=0
\]
for all $\alpha\in \mathcal{N}_t$. We set
\[
p_1(x)=p(x+1)-p(x).
\]
It is easy to check that
\[
p_1(x)=p(x+1)-p(x)=ta_0x^{t-1}+b_1x^{t-2}+\cdots +b_{t-1}
\]
for some $b_1,\ldots,b_{t-1}\in \mathcal{Z}$. We have that
\[
p_1(\alpha)A=0
\]
for all $\alpha\in \mathcal{N}_{t}$. It is clear that
\[
(p_1(\alpha+1)-p_1(\alpha))A=0
\]
for all $\alpha\in \mathcal{N}_{t-1}$. We set
\[
p_2(x)=p_1(x+1)-p_1(x).
\]
It is easy to check that
\[
p_2(x)=t(t-1)a_0x^{t-2}+d_1x^{t-3}+\cdots +d_{t-2}
\]
for some $d_1,\ldots,d_{t-2}\in \mathcal{Z}$. This implies
\[
p_2(\alpha)A=0
\]
for all $\alpha\in \mathbb{N}_{t-1}$. Continuing the same arguments as above we obtain
\[
t!a_0A=0.
\]
Since char$(\mathcal{D})>t$ or char$(\mathcal{D})=0$, and char$(\mathcal{D})\nmid a_0$, we get $A=0$.
\end{proof}

The following technical result will be used in the proof of Theorem \ref{T1}.

\begin{lemma}\label{L2.2}
Let $\mathcal{D}$ be a division ring with char$(\mathcal{D})>t$, $t\geq 1$ or char$(\mathcal{D})=0$. Let $p(x)=a_0x^t+a_1x^{t-1}+\cdots +a_t$ be a polynomial over $\mathcal{D}$ with $a_0\in \mathcal{D}^{\times}$. Let $f:\mathcal{D}\rightarrow \mathcal{D}$ be additive map such that $f(1)=0$. Suppose that $p(a)f(a)=0$ for all $a\in \mathcal{D}$. Then $f=0$.
\end{lemma}

\begin{proof}
By our hypotheses we get
\begin{equation}\label{ee1}
p(a)f(a)=0
\end{equation}
and
\[
(p(a+1)f(a+1)=0
\]
for all $a\in \mathcal{D}$. Since $f(1)=0$ we get from the last relation that
\begin{equation}\label{ee2}
p(a+1)f(a)=0
\end{equation}
for all $a\in \mathcal{D}$. Taking the difference of (\ref{ee2}) and (\ref{ee1}), we get
\[
(p(a+1)-p(a))f(a)=0
\]
for all $a\in \mathcal{D}$. We set
\[
p_1(x)=p(x+1)-p(x).
\]
It is easy to check that
\[
p_1(x)=ta_0x^{t-1}+b_1x^{t-2}+\cdots +b_{t-1}
\]
for some $b_1,\ldots,b_{t-1}\in \mathcal{D}$. We have
\[
p_1(a)f(a)=0
\]
for all $a\in \mathcal{D}$. We have
\[
(p_1(a+1)-p_1(a))f(a)=0
\]
for all $a\in \mathcal{D}$. We set
\[
p_2(x)=p_1(x+1)-p_1(x).
\]
It is easy to check that
\[
p_2(x)=t(t-1)a_0x^{t-2}+d_1x^{t-3}+\cdots +d_{t-2},
\]
where $d_1,\ldots,d_{t-2}\in \mathcal{D}$. We have
\[
p_2(a)f(a)=0
\]
for all $a\in \mathcal{D}$. Continuing the same arguments as above we obtain
\[
t!a_0f(a)=0
\]
for all $a\in \mathcal{D}$. Since char$(\mathcal{D})>t$ or char$(\mathcal{D})=0$, and $a_0\neq 0$, we get $f(a)=0$ for all $a\in \mathcal{D}$.
\end{proof}

The following Hua's identity will be used in the proof of Theorem \ref{T1}.

\begin{lemma}\cite[Lemma 3]{Cat1}\label{LL}
Let $\mathcal{D}$ be a division ring. For any $a,b\in \mathcal{D}$ such that $ab\neq 0,1$:
\begin{equation}\label{H1}
a-aba=(a^{-1}+(b^{-1}-a)^{-1})^{-1}.
\end{equation}
If we alternate choosing $a=1$ and $b=1$ in (\ref{H1}), then we have the following identities:
\begin{equation}\label{H2}
a-a^2=(a^{-1}+(1-a)^{-1})^{-1}\quad\mbox{and}\quad 1-b=(1+(b^{-1}-1)^{-1})^{-1}.
\end{equation}
for $a,b\not\in\{0,1\}$.
\end{lemma}

\section{The proof of Theorem \ref{T1}}

Catalano and Merch$\acute{a}$n \cite{Cat1} gave the following crucial result, by using Lemma \ref{LL}.

\begin{lemma}\cite[Lemma 4]{Cat1}\label{L2.3}
Let $n\geq 2$ and let $\mathcal{D}$ be a division ring. If $f,g:\mathcal{D}\rightarrow \mathcal{D}$ are maps satisfying (\ref{e1}), then
\begin{enumerate}
\item[(a)] $f(b)=\left(-\sum\limits_{k=1}^n\binom{n}{k}(-b)^k+b^n\right)f(1)+g(b)$\quad\mbox{for $b\in \mathcal{D}$}.\\
Additionally, if $n\geq 3$ and char$(\mathcal{D})$ does not divide the sum $\sum\limits_{\substack{j=2\\j\ even}}^{n-1}\binom{n}{j}2^{j+1}$, then
\item[(b)] $f(1)=0$;
\item[(c)] $f(a^2)=\left(-\sum\limits_{k=1}^n\binom{n}{k}(-1)^ka^k+a^n\right)f(a)$ for all $a\in \mathcal{D}$.
\end{enumerate}
\end{lemma}

The following result is a modification of Lemma \ref{L2.3}, which can be proved by using both Lemma \ref{L2.1} and the same arguments as that of Lemma \ref{L2.3}.

\begin{lemma}\label{L2.4}
Let $n\geq 2$ and let $\mathcal{D}$ be a division ring. If $f,g:\mathcal{D}\rightarrow \mathcal{D}$ are maps satisfying (\ref{e1}), then
\begin{enumerate}
\item[(a)] $f(b)=\left(-\sum\limits_{k=1}^n\binom{n}{k}(-b)^k+b^n\right)f(1)+g(b)$\quad\mbox{for $b\in \mathcal{D}$}.\\
Additionally, if $n\geq 3$ and char$(\mathcal{D})>n$ or char$(\mathcal{D})=0$, then
\item[(b)] $f(1)=0$;
\item[(c)] $f(a^2)=\left(-\sum\limits_{k=1}^n\binom{n}{k}(-1)^ka^k+a^n\right)f(a)$ for all $a\in \mathcal{D}$.
\end{enumerate}
\end{lemma}
\begin{proof}
We start by proving (a). Firstly, we notice that the expression in (a) trivially holds
for $b=0,1$. Therefore we may assume that $b\neq 0,1$. Using (\ref{H2}) and the additivity of $f$ and $g$, we have that
\begin{eqnarray}\label{H3}
\begin{split}
f(b)&=f(1)-f(1+(b^{-1}-1)^{-1})^{-1})\\
&=f(1)+(1+(b^{-1}-1)^{-1})^{-n}g(1+(b^{-1}-1)^{-1})\\
&=f(1)+(1-b)^{n})g(1)+b^n(b^{-1}-1)^ng((b^{-1}-1)^{-1})\\
&=f(1)-(1-b)^nf(1)-b^n(b^{-1}-1)^n(b^{-1}-1)^{-n}f(b^{-1}-1)\\
&=f(1)-(1-b)^nf(1)-b^nf(b^{-1})+b^nf(1)\\
&=f(1)-(1-b)^nf(1)+b^nb^{-n}g(b)+b^nf(1)\\
&=f(1)-\sum\limits_{k=0}^n\binom{n}{k}(-b)^kf(1)+g(b)+b^nf(1)\\
&=\left(-\sum\limits_{k=1}^n\binom{n}{k}(-b)^k+b^n\right)f(1)+g(b).
\end{split}
\end{eqnarray}
We proceed to prove (b). Let $x\in \mathcal{D}$. We want to study $f(x+1)$. First,
using equation (\ref{H3}), we have
\begin{eqnarray*}
\begin{split}
f(x+1)&=\left(-\sum\limits_{k=1}^n\binom{n}{k}(x+1)^k+(x+1)^n\right)f(1)+g(x+1)\\
&=\left(-\sum\limits_{k=1}^n\binom{n}{k}(-1)^k\sum\limits_{j=0}^k\binom{k}{j}x^j+\sum\limits_{j=0}^n\binom{n}{j}x^j\right)f(1)+g(x)-f(1).
\end{split}
\end{eqnarray*}
Second, using the additivity of $f$ along with equation (\ref{H3}), we have
\begin{eqnarray*}
\begin{split}
f(x+1)&=f(x)+f(1)\\
&=\left(-\sum\limits_{k=1}^n\binom{n}{k}(-1)^kx^k+x^n\right)f(1)+g(x)+f(1).
\end{split}
\end{eqnarray*}
Therefore, simplifying slightly, we obtain
\begin{equation}\label{H4}
\left(-\sum\limits_{k=1}^n\sum\limits_{j=0}^k\binom{n}{k}\binom{k}{j}(-1)^kx^j+\sum\limits_{j=1}^n\binom{n}{j}(1+(-1)^j)x^j-x^n-1\right)f(1)=0.
\end{equation}

We proceed, in turn, to examine the terms associated to the different monomials $x^j$, with $j=0,\ldots,n$. We note that if $j=n$, then $k=n$ in the previous sums, and we have
\[
-(-1)^nx^n+(1+(-1)^n)x^n-x^n=0,
\]
and so all $x^n$ terms cancel from equation (\ref{H4}). Similarly, all constant terms in equation (6)
cancel; indeed, when $j=0$, using the binomial formula, we have that
\[
-\sum\limits_{k=1}^n\binom{n}{k}(-1)^k-1=0.
\]
For $1\leq j\leq n-1$, in order to examine the terms associated to $x^j$, we analyze
\[
-\sum\limits_{k=j}^n\binom{n}{k}\binom{k}{j}(-1)^kx^j+\binom{n}{j}(1+(-1)^j)x^j.
\]
Since for $r\in\mathbb{Z}$,
\[
\binom{n}{j}(1+r)^{n-j}=\sum\limits_{k=j}^n\binom{n}{k}\binom{k}{j}r^{k-j},
\]
We have that
\[
-\sum\limits_{k=j}^n\binom{n}{k}\binom{k}{j}(-1)^kx^j+\binom{n}{j}(1+(-1)^j)x^j=\binom{n}{j}(1+(-1)^j)x^j.
\]
Additionally, if $j$ is odd, $\binom{n}{j}(1+(-1)^j)x^j=0$.  Therefore, equation (\ref{H4}) simplifies to
\[
2\sum\limits_{\substack{j=2\\j\ even}}^{n-1}\binom{n}{j}x^jf(1)=0.
\]
We set
\[
p(x)=2\sum\limits_{\substack{j=2\\j even}}^{n-1}\binom{n}{j}x^j.
\]
It is clear that
\[
p(x)=\left\{
\begin{aligned}
n(n-1)x^{n-2}&+\cdots +n(n-1)x^2,\quad\mbox{if $n$ is even};\\
2nx^{n-1}&+\cdots +n(n-1)x^2, \quad\mbox{if $n$ is odd}
\end{aligned}
\right.
\]
and
\[
p(\alpha)f(1)=0
\]
for all $\alpha\in\mathcal{N}_{n}$. Since char$(\mathcal{D})>n$ or char$(\mathcal{D})=0$, we get from Lemma \ref{L2.1} that $f(1)=0$. We finally remark that the statement (c) can be proved by the same arguments as in the proof of Lemma \ref{L2.3}(c).
\end{proof}

Using Lemma \ref{L2.4} and the same arguments as that of Theorem \ref{C}, we can give

\begin{proof}[The proof of Theorem \ref{T1}]

In view of Lemma \ref{L2.4}(b), we note that $f(1)=0$. Let $x\in \mathcal{D}^{\times}$. We will calculate $f((x+1)^2)$ in two different ways. First, using Lemma \ref{L2.4}(c), we have that
\begin{eqnarray*}
\begin{split}
f((x+1)^2)&=\left(-\sum\limits_{k=1}^n\binom{n}{k}(-1)^k(x+1)^k+(x+1)^n\right)f(x+1)\\
&=\left(-\sum\limits_{k=1}^n\sum\limits_{j=0}^k\binom{n}{k}\binom{k}{j}(-1)^kx^j+\sum\limits_{j=0}^n\binom{n}{j}x^j\right)f(x).
\end{split}
\end{eqnarray*}
Alternatively, using again (c) in Lemma \ref{L2.4}, we have
\begin{eqnarray*}
\begin{split}
f((x+1)^2)&=f(x^2)+2f(x)\\
&=\left(-\sum\limits_{j=1}^n\binom{n}{j}(-1)^jx^j+x^n\right)f(x)+2f(x).
\end{split}
\end{eqnarray*}
Equating, we can see that we have
\[
\left(-\sum\limits_{k=1}^n\sum\limits_{j=0}^k\binom{n}{k}\binom{k}{j}(-1)^kx^j+\sum\limits_{j=1}^n\binom{n}{j}(1+(-1)^j)x^j-x^n-1\right)f(x)=0.
\]
Using the same arguments as in the proof of Lemma \ref{L2.4}(b) we can obtain
\[
2\sum\limits_{\substack{j=2\\j\  even}}^{n-1}\binom{n}{j}x^jf(x)=0.
\]
We set
\[
p(x)=2\sum\limits_{\substack{j=2\\j\  even}}^{n-1}\binom{n}{j}x^j.
\]
It follows from the last relation that
\[
p(x)f(x)=0
\]
for all $x\in \mathcal{D}$. It is clear that
\[
p(x)=\left\{
\begin{aligned}
n(n-1)x^{n-2}&+\cdots +n(n-1)x^2,\quad\mbox{if $n$ is even};\\
2nx^{n-1}&+\cdots +n(n-1)x^2, \quad\mbox{if $n$ is odd}.
\end{aligned}
\right.
\]
Since char$(\mathcal{D})>n$ or char$(\mathcal{D})=0$, we get from Lemma \ref{L2.2} that $f=0$. This immediately implies that $g=0$, obtaining the desired result.
\end{proof}

\section{The proof of Theorem \ref{T2}}

Using Theorem \ref{T1}, Lemma \ref{L2.1}, and some arguments in the proof of Theorem \ref{T1.3}, we give the following

\begin{proof}[The proof of Theorem \ref{T2}]
By our hypothesis we have
\begin{equation}\label{e3.1}
f(X)+X^ng(X^{-1})=0
\end{equation}
for all $X\in \mathcal{R}^{\times}$. Replacing $X$ by $X^{-1}$ in (\ref{e3.1}), we get
\[
f(X^{-1})+X^{-n}g(X)=0
\]
for all $X\in \mathcal{R}^{\times}$. That is
\begin{equation}\label{e3.2}
g(X)+X^nf(X^{-1})=0
\end{equation}
for all $X\in \mathcal{R}^{\times}$. In view of both (\ref{e3.1}) and (\ref{e3.2}) we note that $f$ and $g$ have the same properties. Let
\[
f(X)=\sum\limits_{1\leq i,j\leq m}F_{ij}(X)E_{ij}
\]
and
\[
g(X)=\sum\limits_{1\leq i,j\leq m}G_{ij}(X)E_{ij}.
\]
Let $\delta_{ij}:\mathcal{D}\rightarrow \mathcal{D}$, $1\leq i,j\leq m$ be an
additive mapping such that $\delta_{ij}(\alpha)=F_{ij}(\alpha I)$ for all $\alpha\in \mathcal{D}$. Let $\tau_{ij}:\mathcal{D}\rightarrow \mathcal{D}$, $1\leq i,j\leq m$ be an
additive mapping such that $\tau_{ij}(\alpha)=G_{ij}(\alpha I)$ for all $\alpha\in \mathcal{D}$.  We have
\[
f(\alpha I)+\alpha^ng(\alpha^{-1} I)=0
\]
for all $\alpha\in \mathcal{D}^{\times}$. Hence
\[
\delta_{ij}(\alpha)+\alpha^n\tau_{ij}(\alpha^{-1})=0
\]
for all $\alpha\in \mathcal{D}^{\times}$. Thus, in view of Theorem \ref{T1}, we have $\delta_{ij}=\tau_{ij}=0$ for $1\leq i,j\leq m$. Hence $f(\alpha I)=0=g(\alpha I)$ for all $\alpha\in \mathcal{D}$. We now claim that $f=0=g$ by the following several steps:\\

\emph{Step 1.} We claim that
\[
f(X^2)=\left(\sum\limits_{k=1}^{n-1}\binom{n}{k}X^k\right)f(X)
\]
and
\[
g(X^2)=\left(\sum\limits_{k=1}^{n-1}\binom{n}{k}X^k\right)g(X)
\]
for all $X\in \mathcal{R}^{\times}$ with $X+I\in \mathcal{R}^{\times}$.

For $X\in \mathcal{R}^{\times}$ with $X+I\in \mathcal{R}^{\times}$, we note that
\[
(X^{-1}-(X+I)^{-1})^{-1}=X(X+I).
\]
In view of (\ref{e3.2}) we have
\begin{eqnarray*}
\begin{split}
0&=g(X^{-1}-(X+I)^{-1})+(X^{-1}-(X+I)^{-1})^nf(X(X+I))\\
&=g(X^{-1}-(X+I)^{-1})+(X(X+I))^{-n}(f(X^2)+f(X))\\
&=g(X^{-1})-g((X+I)^{-1})+(X(X+I))^{-n}f(X^2)+(X(X+I))^{-n}f(X)\\
&=-X^{-n}f(X)+(X+I)^{-n}f(X+I)+(X(X+I))^{-n}f(X^2)+(X(X+I))^{-n}f(X),\\
&=(-X^{-n}+(X+I)^{-n}+(X(X+I))^{-n})f(X)+(X(X+I))^{-n}f(X^2),
\end{split}
\end{eqnarray*}
which implies
\begin{eqnarray*}
\begin{split}
f(X^2)&=-(X(X+I))^{n}(-X^{-n}+(X+I)^{-n}+(X(X+I))^{-n})f(X)\\
&=(((X+I)^n-X^n-I)f(X)\\
&=\left(\sum\limits_{k=1}^{n-1}\binom{n}{k}X^k\right)f(X).
\end{split}
\end{eqnarray*}
Similarly, we get from (\ref{e3.1}) that
\[
g(X^2)=\left(\sum\limits_{k=1}^{n-1}\binom{n}{k}X^k\right)g(X).
\]

\emph{Step 2.} We claim that
\[
f(X)=\frac{1}{3}\left(\sum\limits_{k=1}^{2n}\binom{2n}{k}\right)Xf(X)
\]
and
\[
g(X)=\frac{1}{3}\left(\sum\limits_{k=1}^{2n}\binom{2n}{k}\right)Xg(X)
\]
for all $X^2=X\in \mathcal{R}$. In particular, $f(E_{ii})=0=g(E_{ii})$, $i=1,\ldots,m$.

Let $X\in \mathcal{R}$ be an idempotent. Then $(I+X)^{-1}=I-\frac{1}{2}X$. In view of (\ref{e3.1}) we
have
\[
f(I+X)+(I+X)^ng(I-\frac{1}{2}X)=0,
\]
this implies
\begin{equation}\label{e3.3}
f(X)-\frac{1}{2}(I+X)^ng(X)=0.
\end{equation}
Similarly, we get from (\ref{e3.2}) that
\begin{equation}\label{e3.4}
g(X)-\frac{1}{2}(I+X)^nf(X)=0.
\end{equation}
It follows from both (\ref{e3.3}) and (\ref{e3.4}) that
\begin{eqnarray*}
\begin{split}
f(X)&=(\frac{1}{2}(I+X)^n)^2f(X)\\
&=\frac{1}{4}\left(\sum\limits_{k=0}^{2n}\binom{2n}{k}X^k\right)f(X)\\
&=\frac{1}{4}\left(\sum\limits_{k=1}^{2n}\binom{2n}{k}X+I\right)f(X).
\end{split}
\end{eqnarray*}
Since char$(\mathcal{D})>n\geq 3$ or char$(\mathcal{D})=0$, we get from the last relation that
\[
f(X)=\frac{1}{3}\left(\sum\limits_{k=1}^{2n}\binom{2n}{k}\right)Xf(X).
\]
Similarly, we have
\[
g(X)=\frac{1}{3}\left(\sum\limits_{k=1}^{2n}\binom{2n}{k}\right)Xg(X).
\]
In particular
\[
0=f(I)=\sum\limits_{i=1}^mf(E_{ii})=\frac{1}{3}\sum\limits_{i=1}^m\left(\sum\limits_{k=1}^{2n}\binom{2n}{k}\right)E_{ii}f(E_{ii}).
\]
We get that
\[
f(E_{ii})=\frac{1}{3}\left(\sum\limits_{k=1}^{2n}\binom{2n}{k}\right)E_{ii}f(E_{ii})=0
\]
for $i=1,\ldots,m$. Similarly, we have
\[
g(E_{ii})=0
\]
for $i=1,\ldots,m$.

\emph{Step 3.} We claim that $Xf(X)=0=Xg(X)$ for all $X\in\mathcal{R}$ with $X^2=0$.

Note that $(I+X)^{-1}=I-X$. In view of (\ref{e3.1}) we
have
\[
f(I+X)+(I+X)^{n}g(I-X)=0,
\]
this implies
\begin{equation}\label{e3.6}
f(X)-(I+nX)g(X)=0.
\end{equation}
Similarly, in view of (\ref{e3.2}) we get
\begin{equation}\label{e3.7}
g(X)-(1+nX)f(X)=0.
\end{equation}
We get from both (\ref{e3.6}) and (\ref{e3.7}) that
\[
f(X)-(I+nX)^2f(X)=0.
\]
This implies
\[
2nXf(X)=0.
\]
Since char$(\mathcal{D})>n$ or char$(\mathcal{D})=0$, we obtain $Xf(X)=0$. Similarly, we have $Xg(X)=0$.

\emph{Step 4.} For any $\alpha,\beta\in \mathcal{D}$, we set $\alpha\circ\beta=\alpha\beta+\beta\alpha$. We claim that
\[
f((\alpha\circ \beta)E_{ij})=\left(\sum\limits_{k=1}^{n-1}\binom{n}{k}\alpha^k\right)f(\beta E_{ij})
\]
and
\[
g((\alpha\circ \beta)E_{ij})=\left(\sum\limits_{k=1}^{n-1}\binom{n}{k}\alpha^k\right)g(\beta E_{ij})
\]
for all $\alpha,\beta\in \mathcal{D}$ and $i\neq j$ with $\alpha\neq -1,0$.

We set
\[
X=\alpha I+\beta E_{ij}.
\]
Then $X$, $X+I$ are invertible in $\mathcal{R}$.  Since $(\beta E_{ij})^2=0$ for $i\neq j$, using Step 3, we have
\[
E_{ij}f(\beta E_{ij})=0.
\]
Thus, invoking Step 1, we get
\begin{eqnarray*}
\begin{split}
f(X^2)&=\left(\sum\limits_{k=1}^{n-1}\binom{n}{k}X^k\right)f(X)\\
&=\left(\sum\limits_{k=1}^{n-1}\binom{n}{k}(\alpha I+\beta E_{ij})^k\right)f(\alpha I+\beta E_{ij})\\
&=\left(\sum\limits_{k=1}^{n-1}\binom{n}{k}(\alpha^k I+\gamma E_{ij})\right)f(\beta E_{ij})\\
&=\left(\sum\limits_{k=1}^{n-1}\binom{n}{k}\alpha^k\right)f(\beta E_{ij}),
\end{split}
\end{eqnarray*}
for some $\gamma\in \mathcal{D}$. Also
\begin{eqnarray*}
\begin{split}
f(X^2)&=f((\alpha I+\beta E_{ij})^2)\\
&=f(\alpha^2 I+(\alpha\circ\beta)E_{ij})\\
&=f((\alpha\circ\beta)E_{ij}).
\end{split}
\end{eqnarray*}
Thus, we get
\[
f(\alpha\circ \beta)E_{ij})=\left(\sum\limits_{k=1}^{n-1}\binom{n}{k}\alpha^k\right)f(\beta E_{ij}).
\]
Similarly, we get
\[
g(\alpha\circ \beta)E_{ij})=\left(\sum\limits_{k=1}^{n-1}\binom{n}{k}\alpha^k\right)g(\beta E_{ij}).
\]

\emph{Step 5.} We claim that $f(\mathcal{D}E_{ii})=0=g(\mathcal{D}E_{ii})$ for $i=1,\ldots,m$.

For any $\alpha\in \mathcal{D}$, suppose first that $\alpha\in \mathcal{Z}\cdot 1$. In view of Step 2 we note that
\[
f(\alpha E_{ii})=\alpha f(E_{ij})=0
\]
for $i=1,\ldots,m$. Suppose next that $\alpha\not\in \mathcal{Z}\cdot 1$.  Since both $I+\alpha E_{ii}$ and $2I+\alpha E_{ii}$ are invertible in $\mathcal{R}$, and $f(I)=0$,
by Step 1, we get
\begin{eqnarray*}
\begin{split}
f((I+\alpha E_{ii})^2)&=\left(\sum\limits_{k=1}^{n-1}\binom{n}{k}(I+\alpha E_{ii})^k\right)f(I+\alpha E_{ii})
\end{split}
\end{eqnarray*}
implying
\[
2f(\alpha E_{ii})+f(\alpha^2E_{ii})=\left(\sum\limits_{k=1}^{n-1}\binom{n}{k}(I+\alpha E_{ii})^k\right)f(\alpha E_{ii}).
\]
Thus, we have
\begin{equation}\label{e3.8}
f(\alpha^2E_{ii})=\left(\sum\limits_{k=1}^{n-1}\binom{n}{k}(I+\alpha E_{ii})^k-2I\right)f(\alpha E_{ii}).
\end{equation}
Since $\alpha\not\in \mathcal{Z}\cdot 1$, clearly $\alpha+1\not\in \mathcal{Z}\cdot 1$. Replacing $\alpha$ by $\alpha+1$ in (\ref{e3.8}), we get
\begin{equation}\label{e3.9}
f((\alpha+1)^2E_{ii})=\left(\sum\limits_{k=1}^{n-1}\binom{n}{k}(I+(\alpha+1) E_{ii})^k-2I\right)f((\alpha+1) E_{ii}).
\end{equation}
Note that $f(E_{ii})=0$. Taking the difference of (\ref{e3.9}) and (\ref{e3.8}), we get
\begin{eqnarray*}
\begin{split}
2f(\alpha E_{ii})&=\left(\sum\limits_{k=1}^{n-1}\binom{n}{k}((I+(\alpha+1) E_{ii})^k-(I+\alpha E_{ii})^k)\right)f(\alpha E_{ii})\\
&=\left(\sum\limits_{k=1}^{n-1}\binom{n}{k}\left(\sum\limits_{s=0}^k\binom{k}{s}(\alpha+1)^{s}E_{ii}^s
-\sum\limits_{s=0}^k\binom{k}{s}\alpha^{s}E_{ii}^s\right)\right)f(\alpha E_{ii})\\
&=\left(\sum\limits_{k=1}^{n-1}\binom{n}{k}\left(\sum\limits_{s=1}^k\binom{k}{s}\left((\alpha+1)^s-\alpha^s\right)\right)\right)E_{ii}f(\alpha E_{ii}),
\end{split}
\end{eqnarray*}
this implies
\[
f(\alpha E_{ii})=\frac{1}{2}\left(\sum\limits_{k=1}^{n-1}\binom{n}{k}\left(\sum\limits_{s=1}^k\binom{k}{s}((\alpha+1)^s-\alpha^s)\right)\right)E_{ii}f(\alpha E_{ii}).
\]
Thus, we have
\begin{eqnarray*}
\begin{split}
0&=f(\alpha I)=\sum\limits_{i=1}^mf(\alpha E_{ii})\\
&=\frac{1}{2}\sum\limits_{i=1}^m\left(\sum\limits_{k=1}^{n-1}\binom{n}{k}\left(\sum\limits_{s=1}^k\binom{k}{s}((\alpha+1)^s-\alpha^s)\right)\right)E_{ii}f(\alpha E_{ii}).
\end{split}
\end{eqnarray*}
This implies that
\[
f(\alpha E_{ii})=\frac{1}{2}\left(\sum\limits_{k=1}^{n-1}\binom{n}{k}\left(\sum\limits_{s=1}^k\binom{k}{s}((\alpha+1)^s-\alpha^s)\right)\right)E_{ii}f(\alpha E_{ii})=0
\]
for all $i=1,\ldots,m$. Similarly, we have $g(\alpha E_{ii})=0$ for all $i=1,\ldots,m$.

\emph{Step 6.} For any $\alpha\in F$ with $\alpha\neq -1,0$, where $F$ is the centre of $\mathcal{D}$, $\beta\in \mathcal{D}$, and $i\neq j$, we claim that
\[
f(\alpha \beta E_{ij})=\left(\frac{1}{2}\sum\limits_{k=1}^{n-1}\binom{n}{k}\alpha^k\right)f(\beta E_{ij})
\]
and
\[
g(\alpha \beta E_{ij})=\left(\frac{1}{2}\sum\limits_{k=1}^{n-1}\binom{n}{k}\alpha^k\right)g(\beta E_{ij}).
\]
In view of Step 4, we have
\[
f((\alpha\circ\beta) E_{ij})=\left(\sum\limits_{k=1}^{n-1}\binom{n}{k}\alpha^k\right)f(\beta E_{ij})
\]
for all $\alpha \in F$ and $\beta\in \mathcal{D}$. Recall that char$(\mathcal{D})\neq 2$. Hence
\[
f((\alpha\beta) E_{ij})=\frac{1}{2}\left(\sum\limits_{k=1}^{n-1}\binom{n}{k}\alpha^k\right)f(\beta E_{ij}).
\]
Similarly, we have
\[
g((\alpha\beta) E_{ij})=\frac{1}{2}\left(\sum\limits_{k=1}^{n-1}\binom{n}{k}\alpha^k\right)g(\beta E_{ij}).
\]

\emph{Step 7.} We claim $f(X)=0=g(X)$ for all $X\in M_m(F)$, where $F$ is the centre of $\mathcal{D}$.

For any $X=\sum\limits_{1\leq i,j\leq m}\alpha_{ij}E_{ij}\in M_m(F)$, we set
\[
\left\{
\begin{aligned}
\beta_{ij}&=\alpha_{ij},\quad\mbox{if $\alpha_{ij}=-1,0$};\\
\beta_{ij}&=\frac{1}{2}\sum\limits_{k=1}^{n-1}\binom{n}{k}\alpha_{ij}^k,\quad\mbox{if $\alpha_{ij}\neq -1$ and $\alpha_{ij}\neq 0$}.
\end{aligned}
\right.
\]
In view of Step 6, we get
\[
f(X)=\sum\limits_{1\leq i,j\leq m}f(\alpha_{ij}E_{ij})=\sum\limits_{1\leq i,j\leq m}\beta_{ij}f(E_{ij}).
\]
Note that $f(E_{ii})=0$ for all $i=1,\ldots,m$. It is sufficient to show $f(E_{ij})=0$ for all $i\neq j$.

For any $\alpha\in \mathcal{N}_{n-1}$, we note that $\alpha I+E_{ij}$ and $(\alpha+1)I+E_{ij}$ are invertible in $\mathcal{R}$. In view of Step 1, we get
\begin{eqnarray*}
\begin{split}
f((\alpha I+E_{ij})^2)&=\left(\sum\limits_{k=1}^{n-1}\binom{n}{k}(\alpha I+E_{ij})^k\right)f(\alpha I+E_{ij})\\
&=\left(\sum\limits_{k=1}^{n-1}\binom{n}{k}(\alpha^k I+\alpha^{k-1}kE_{ij})\right)f(E_{ij}),
\end{split}
\end{eqnarray*}
this implies
\[
2\alpha f(E_{ij})=\left(\sum\limits_{k=1}^{n-1}\binom{n}{k}(\alpha^k I+\alpha^{k-1}kE_{ij})\right)f(E_{ij}).
\]

In view of Step 3, we note that $E_{ij}f(E_{ij})=0$. We get from the last relation that
\[
\left(\sum\limits_{k=1}^{n-1}\binom{n}{k}\alpha^k-2\alpha\right)f(E_{ij})=0.
\]
Since char$(\mathcal{D})>n$ or char$(\mathcal{D})=0$ we get from the last relation that
\[
\left(\sum\limits_{k=1}^{n-1}\binom{n}{k}\alpha^{k-1}-2\right)f(E_{ij})=0.
\]
We set
\[
p(x)=\sum\limits_{k=1}^{n-1}\binom{n}{k}x^{k-1}-2=nx^{n-2}+\cdots +(-2).
\]
It follows that
\[
p(\alpha)f(E_{ij})=0
\]
for all $\alpha\in \mathcal{N}_{n-1}$. In view of Lemma \ref{L2.1} we get that $f(E_{ij})=0$. Hence $f(X)=0$. Similarly, we can obtain that $g(X)=0$.

\emph{Step 8.} We claim that $f(\mathcal{D}E_{ij})=0=g(\mathcal{D}E_{ij})$ for all $i\neq j$.

Let $\alpha\in \mathcal{D}^{\times}$. Let $X\in\{E_{ij},-E_{ij}\}$. Note that
\[
(\alpha I+X)^{-1}=\alpha^{-1}I-\alpha^{-2}X.
\]
Invoking our hypothesis, we get
\[
f(\alpha I+X)+(\alpha I+X)^ng(\alpha^{-1}I-\alpha^{-2}X)=0.
\]
Since $f(X)=0$, the above equation reduces to
\[
(\alpha I+X)^ng(-\alpha^2X)=0,
\]
this implies
\begin{equation}\label{e3.12}
(\alpha I+nX)g(\alpha^{-2}X)=0.
\end{equation}
Replacing $X$ by $-X$ in (\ref{e3.12}), we have
\begin{equation}\label{e3.13}
(\alpha I-nX)g(\alpha^{-2}X)=0.
\end{equation}
Adding (\ref{e3.12}) and (\ref{e3.13}),  we have $2\alpha g(\alpha^{-2}X)=0$ and so $g(\alpha^{-2}X)=0$. That is, $g(\beta^2 X)=0$ for all $\beta\in \mathcal{D}$. Replacing $\beta$ by $\beta+1$, we get
\[
g((\beta^2+2\beta+1)X)=0.
\]
Since $g(X)=0$ and $g(\beta^2X)=0$ we have $2g(\beta X)=0$ and so $g(\beta X)=0$.  we get $g(\mathcal{D}E_{ij})=0$ for any $1\leq i,j\leq m$. Similarly, $f(\mathcal{D}E_{ij})=0$ for any $1\leq i,j\leq m$.

Up to now we have proved that $f(\mathcal{D}E_{ij})=0=g(\mathcal{D}E_{ij})$ for any $1\leq i,j\leq m$. Note that we can write $\mathcal{R}=\sum\limits_{1\leq i,j\leq m}\mathcal{D}E_{ij}$. We get $f(\mathcal{R})=0=g(\mathcal{R})$, as desired. This completes the proof of the result.
\end{proof}

\end{document}